\documentclass[11pt,bezier,epsf]{article}
\usepackage{graphicx}
\usepackage{amsmath}
\usepackage{latexsym}
\usepackage{subfigure}
\usepackage{relsize}
\usepackage{amsthm}
\usepackage{amsmath, amsthm, amscd, amsfonts, amssymb}
\usepackage{graphicx,color,amsmath}
\usepackage{colortbl}
\def\bkR{{\rm I\kern-.17em R}}
\def\bkC{{\rm ^{_|}\kern-.47em C}}

\newtheorem{thm}{Theorem}[section]

\newtheorem{lem}[thm]{Lemma}

\newtheorem{defn}[thm]{Definition}

\def\@seccntformat#1{\csname the#1\endcsname.\quad}
\def\numberline#1{\hb@xt@\@tempdima{#1\if&#1&\else.\fi\hfil}}

\newcommand{\abs}[1]{\left\vert#1\right\vert}

\begin{document}
\title{
{\bf \Large  Convergence and non-negativity preserving of the solution of  balanced  method for the delay CIR
model with jump} \vspace{1cm} }
\author{A.S. Fatemion Aghda \thanks{ as.fatemion@modares.ac.ir}, \hspace{0.3cm}
Seyed Mohammad Hosseini \thanks{ hossei\_m@modares.ac.ir},
\hspace{0.3cm}
Mahdieh Tahmasebi,
 \thanks{ 
 tahmasebi@modares.ac.ir}}
\date{Department of Applied Mathematics, Tarbiat Modares University,\\
P.O. Box 14115-175, Tehran, IRAN.}
\maketitle
\begin{abstract}
In this work, we propose the balanced implicit method (BIM) to approximate the solution of  the delay Cox-Ingersoll-Ross (CIR) model with jump which often gives rise to model an asset price and stochastic volatility dependent on past data.  We show that this method 
  preserves
  non-negativity property of the solution of this model with appropriate control functions. We prove the strong convergence
   and investigate the $p$th moment boundedness of the solution of BIM. Finally, we illustrate those results in the last section.
\end{abstract}


$\bf{Keywords }$:  stochastic delay differential equation (SDDE) with jump, the delay CIR model with jump,   balanced method, convergence, non-negativity, moment boundedeness.
\section{Introduction}
 Consider $(\Omega,\mathcal{F},P)$ as a complete probability
 space with right continuous filtration $\{\mathcal{F}_t\}_{t\geq 0}$  while $\mathcal{F}_0$ contains all $P$-null sets.
We consider the delay CIR model with jump introduced by Jiang, Shen and Wu \cite{ji1} 
\begin{equation}\label{eju1}
\left\{\begin{array}{l}
dS(t)=\lambda (\mu -S(t))dt+\sigma S(t-\tau)^{\gamma}\sqrt{S(t)}dW(t)+\delta S(t)d\tilde{N}(t),~~t\geq 0,\\
S(t)=\xi(t),~~t\in [-\tau,0],
\end{array}\right.
\end{equation}
where $\gamma,$ $\lambda$, $\mu$ and $\sigma$ are
 positive constants, $W(t)$ is a standard Brownian motion and  $\tilde{N}(t)=N(t)-\beta t$ is a compensated Poisson process, in which $N(t)$ is a Poisson process with intensity $\beta$. The positive initial value $\xi$ is an $\mathcal{F}_0$-measurable
  $C([-\tau,0];\mathbb{R}^{+})$-valued random variable satisfying
\begin{equation}\label{xi}
E[\sup _{-\tau\leq t\leq0} \abs{\xi(t)}^p]<+\infty ,
\end{equation}
for any $p>0$. 
In particular, the CIR  model (\ref{eju1}) without jump and delay, $\tau=0,\,\delta=0,$ was introduced by Cox, Ingersoll and Ross \cite{r4}, as a model for stochastic volatility, interest rate and other financial quantities. Also, the CIR  model (\ref{eju1}) without jump, $\delta=0$, was introdued by
Wu, et al. \cite{r1} with regard to the fact that stock prices depend on past behaviors. (See also  \cite{ak1, sl1}). 
Unfortunately, SDDEs with jumps have no explicit solution. Thus, constructing an appropriate numerical method to approximate and study the properties of the true solutions of these models are essential. Furthermore, in recent years, researchers are interested in numerical methods satisfying the same properties of the solutions such as positivity.\\
Strong convergence  for stochastic differential equations (SDEs) with jumps is studied in some literatures  \cite{{ju6},{ju3},{ju5},{ju4},{bu1},{ju7}} and strong convergence for the mean-reverting square root process with jump is discussed in \cite{ju13}. There are also some works concerned with positivity of numerical methods of SDEs; for example, see \cite{{hgw},{r6},{r7},{r10},{ro1},{r2}}.\\
For the CIR model (\ref{eju1}), with $\tau=0, \delta=0$, Dereich et al. \cite{der1} investigated the drift non-negativity preserving of implicit Euler method and in 2013, Higham et al. \cite{r5} introduced a new implicit Milstein scheme which preserves non-negativity of solution. Recently, Yang and Wang \cite{jj1} showed that the backward Euler scheme preserves positivity for the CIR model with jump. \\
In this manuscript, we are interested to balanced implicit scheme as a numerical method in order to obtain the positivity of our approximation process. Non-negativity preserving of BIM for SDEs without jumps is well studied (see; e.g. \cite{r9,r3}), and of SDEs with jumps is discussed in \cite{2017, jj1}, under an appropriate choice of control functions. Also, Tan et al. \cite{tan1} showed that the BIM preserves positivity for the stochastic age-dependent population equation.  \\
Strong convergence for SDDEs with jumps is studied in some literatures \cite{{ju11},{ju2},{ja1},{ju10},{ju9},{li1}} and in \cite{tan} for SDDEs without jumps. Wu et al. \cite{r1}, showed the existence of non-negative solution of the delay CIR model without jump and Jiang, et al. \cite{ji1} proved it for the delay mean-reverting
square root process with jump (\ref{eju1}). Also, they showed the Euler Maruyama method converges strongly to the solution and proved  the boundedness of the $p$th moments of the solution to the model and the method. Fatemioon et al. \cite{tahmasebi} investigated strong convergence of BIM for the CIR model and showed that the scheme preserves positivity. \\
To the best of our knowledge, there is no positivity preserving result of numerical method for SDDEs with jumps. The aim of this paper is to preserve positivity of BIM for SDDEs (the delay CIR models) with jumps (\ref{eju1}). To do this, we can not examine traditional control functions used in BIM for SDEs to reach the positivity of BIM for these SDDEs, for instance, see \cite{hgw, 2017}. We define a new appropriate control function and prove that the non-negative solution of the BIM  converges to the solution of the model (\ref{eju1}) in the strong sense. Also, we show the boundedness of $p$-moments of the method for any $p >0$.\\
The paper is organized as follows. In Section 2, we propose the BIM for the SDDE with jump (\ref{eju1}) and choose the especial control functions that the method preserves non-negativity of the solution of the model. Also,  we introduce the continuous case of the method to prove convergence in  next section. In Section 3,  we prove the convergence of the BIM applying to the model  (\ref{eju1}). Some numerical experiments  in last section illustrate the obtained theoretical results of this paper.
\section{Introduction of BIM and its properties}
In this section, we describe the balanced method to approximate the solution of the  delay CIR model with jump (\ref{eju1}). Then, we state the non-negativity preserving concept of solution of numerical methods for this model, based on definitions in \cite{r3}. Also, we investigate the properties of $p$-moments for  the balanced method in continuous time, which we need in the next section. 
\subsection{BIM and non-negativity preserving of method in discrete case }
 Set a uniform mesh on $[0,T]$,  $t_n=nh$, $n=0,...,N$,
$ N\in \mathbb{N}$ for a step size $h\in (0,1)$ as  $h=\frac{\tau}{m}$, for a positive integer $m.$  We introduce  the BIM  for SDDE with jump  (\ref{eju1}) by $s_n=\xi _n=\xi(t_n)$ for $n=-m,-m+1,...,0$ and for $n\geq 0,$
\begin{equation}\label{eju31}
s_{n+1}=s_{n}+\lambda (\mu -s_n)h+\sigma s_{n-m}^{\gamma} \sqrt{s_{n}}\Delta W_n +\delta s_n \Delta \tilde{N}_n+ C_n (s_{n}-s_{n+1}),
\end{equation}
where $C_n=C_0(s_{n},s_{n-m})h+C_1(s_{n},s_{n-m})\abs{\Delta W_n}+C_2(s_{n},s_{n-m})\abs{\Delta \tilde{N}_n}$, such that for control functions $C_0(s_{n},s_{n-m}),$ $C_1(s_{n},s_{n-m})$ and $C_2(s_{n},s_{n-m})$,  the expression $(1+C_0(s_{n},s_{n-m})h+C_1(s_{n},s_{n-m})$\\ $\abs{\Delta W_n}
+C_2(s_{n},s_{n-m})\abs{\Delta \tilde{N}_n})^{-1}$ always exists and is uniformly bounded.\\
The control functions for the BIM (\ref{eju31}) that ensure preserving non-negativity of the solution of delay CIR model with jump (\ref{eju1}) are
\begin{equation}\label{aju5}
C_0(s_n,s_{n-m})=C_0\geq \lambda,
\end{equation}
\begin{equation}\label{eju30}
C_1(s_n,s_{n-m})=
\left\{\begin{array}{l}
\sigma s_{n-m}^{\gamma}\epsilon ^{ -\frac{1}{2}},~~~~~s_n<\epsilon, \\
\sigma\frac{ s_{n-m}^{\gamma}}{\sqrt{s_n}},~~~~ s_n\geq \epsilon,
\end{array}\right.
\end{equation}
\begin{equation}\label{ajum5}
C_2(s_n,s_{n-m})=C_2\geq \delta,
\end{equation}
where $C_0,\,C_2$ are  positive constants.
\begin{defn}\label{d1}
Let $s_n$ be a numerical solution which is computed by a numerical method for solving SDDE with jump (\ref{eju1}). The numerical solution $s_n$ is said to be eternal life time if
\begin{equation}\label{e26}
P(s_n\geq 0|\xi _n \geq 0)=1,~~~\text{for all}~n\geq -m.
\end{equation}
If (\ref{e26}) does not hold, then the numerical solution is said to be finite life time.
\end{defn}
\begin{defn}\label{d2}
Let $s_n$ be a numerical solution which is computed by a numerical method for solving SDDE with jump (\ref{eju1}). The numerical solution $s_n$ is said to be $\epsilon$-life time if
\begin{equation}\label{e26j}
P(s_{n+1}\geq 0|s_n\geq \epsilon ,~s_{n-m}\geq 0)=1,~~~\text{for some}~\epsilon>0.
\end{equation}
\end{defn}
\begin{thm}\label{l4}
The solution of the BIM  (\ref{eju31}) with control functions (\ref{aju5}), (\ref{eju30}) and (\ref{ajum5})   is $\epsilon$-life time.
\end{thm}
\begin{proof}
Assume that $s_n\geq\epsilon ,~s_{n-m}\geq 0.$ According to the BIM  (\ref{eju31}) with control functions (\ref{aju5}), (\ref{eju30}) and (\ref{ajum5}), we have
\begin{align}\label{eju32}
s_{n+1}&=s_{n}+\frac{\lambda (\mu - s_n)h+\sigma s_{n-m}^{\gamma}\sqrt{s_{n}}\Delta W_n+\delta s_n\Delta \tilde{N}_n}{1+C_0\,h+\sigma \frac{s_{n-m}^{\gamma}}{\sqrt{s_{n}}}\abs{\Delta W_n}+C_2\abs{\Delta \tilde{N}_n}}\nonumber\\
&=s_{n}\left(\frac{(1+C_0\,h-\lambda \,h)\sqrt{s_{n}}+\sigma s_{n-m}^{\gamma} (\Delta W_n+\abs{\Delta W_n})+(C_2\abs{\Delta \tilde{N}_n}+\delta \Delta \tilde{N}_n){\sqrt{s_{n}}}}{(1+C_0\,h)\sqrt{s_{n}}+\sigma s_{n-m}^{\gamma}\abs{\Delta W_n}+C_2\sqrt{s_n}\abs{\Delta \tilde{N}_n} }\right)\nonumber\\
&~~~~+\frac{\lambda \mu h}{1+C_0\,h+\sigma \frac{s_{n-m}^{\gamma}}{\sqrt{s_{n}}}\abs{\Delta W_n}+C_2\abs{\Delta \tilde{N}_n}}.\nonumber
\end{align}
So, it is clear that $s_{n+1}\geq0.$
\end{proof}
\subsection{BIM and boundedness of the $p$-moments in continuous case}
It is more convenient to use the time-continuous approximation of the BIM (\ref{eju31}) as 
\begin{equation}\label{eju6j}
s(t)=
\left\{\begin{array}{l}
\xi(t),~~~-\tau\leq t\leq 0,\\
\xi(0)+\lambda \int _{0}^{t}\frac{(\mu -\hat{s}(r))}{1+C_r(\hat{s}(r),\hat{s}(r-\tau))}dr+\sigma \int _{0}^{t}\frac{\hat{s}^{\gamma}(r-\tau)\sqrt{\hat{s}(r)}}{1+C_r(\hat{s}(r),\hat{s}(r-\tau))}dW(r)+\delta \int _{0}^{t} \frac{\hat{s}(r)}{1+C_r(\hat{s}(r),\hat{s}(r-\tau))}d\tilde{N}(r) ,~~t\geq 0,
\end{array}\right.
\end{equation}
where
\begin{equation}\label{e5}
\hat s(t)=
\left\{\begin{array}{l}
\xi(t),~~~-\tau\leq t\leq 0,\\
\sum _{n=0}^{[\frac{T}{h}]}s_{n} 1_{[n h, (n+1) h)}(t), ~~~t\geq 0,
\end{array}\right.
\end{equation}
with $[\frac{T}{h}]$ as the integer part of $\frac{T}{h}$, $C_r(\hat{s}(r),\hat{s}(r-\tau))=C_0(\hat{s}(r),\hat{s}(r-\tau))h+C_1(\hat{s}(r),\hat{s}(r-\tau))
\abs{\Delta{W}(r)}+C_2(\hat{s}(r),\hat{s}(r-\tau))\abs{\Delta{\tilde{N}}(r)}$, in which $C_0,C_1,C_2$ defined  in (\ref{aju5}),  (\ref{eju30}) and (\ref{ajum5}) and $\Delta{W}(r)=W(t_
{k+1})-W(t_k)$ and $\Delta{\tilde{N}}(r)=\tilde{N}(t_
{k+1})-\tilde{N}(t_k)$ for $r\in [t_k,t_{k+1})$.
For simplicity of notation, we set $C_r:=C_r(\hat{s}(r),\hat{s}(r-\tau))$.\\
It is easy to observe that  $s(nh)=s_{n},$  so an error bound for $s(t)$  will automatically
imply an error bound for $s_n$. Also, it is  easy to obtain the
following natural relationship
\begin{equation}\label{e13}
\sup _{0 \leq t \leq T}\abs{\hat{s}(t)}\leq \sup _{0 \leq t \leq T}\abs{{s}(t)}.
\end{equation}
Now, we study the $p$th moment properties of the balanced method.
\begin{thm}\label{ljuj1}
There exists a constant $K_1$, which is independent of $h$, such that
\begin{equation}\label{e7}
E(\sup _{-\tau \leq t \leq T} \abs {s(t)}^p)\leq K_1,
\end{equation}
holds for $p>2,$ and
\begin{equation}
E\abs{s(t)}^p\leq E[\abs{s(t)}^3]^{\frac{p}{3}}\leq K_1^{\frac{p}{3}},
\end{equation}
holds for  $0<p\leq 2$.
\end{thm}
\begin{proof}
Define the stopping time, for any $k>0,$
\begin{equation*}
\tau _k=T\wedge \inf \left\{t\geq 0, \abs{s(t)}>k\right\},
\end{equation*}
We set $\inf{\emptyset}=\infty$, where $\emptyset$ denotes the empty set. For any $t_1\in[0,T]$,  from the H$\ddot{\text{o}}$lder
  inequality and the Burkholder-Davis-Gundy inequality \cite{r16} and applying
 the fact   $\frac{1}{1+C_r}\leq 1$ and the relation (\ref{eju6j}), we result that there exist positive constants $C_p$ and $C_{p,\beta}$ such that
\begin{align}\label{ej9}
E(\sup _{0 \leq t \leq t_1}\abs{s(t\wedge \tau _k)}^p)& \leq 4^{p-1}\left[E\abs{\xi(0)}^p+\lambda ^pT^{p-1}E\int _{0}^{t_1\wedge \tau _k}\abs{\frac{(\mu -\hat{s}(r))}{1+C_r}}^pdr\right.\nonumber\\
& ~~+\sigma ^p E\left[\sup _{0\leq t\leq t_1}\abs{\int _{0}^{t\wedge \tau _k}\frac{\hat{s}(r-\tau)^{\gamma}\sqrt{\hat{s}(r)}}{1+C_r}dW(r)}^p\right]\nonumber\\
& ~~\left. + \delta ^pE\left[\sup _{0\leq t\leq t_1}\abs{\int _{0}^{t\wedge \tau _k}\frac{\hat{s}(r)}{1+C_r}d\tilde{N}(r)}^p\right]\right]\nonumber\\
&\leq 4^{p-1}\left[E\abs{\xi(0)}^p+\lambda ^pT^{p-1}E\int _{0}^{t_1\wedge \tau _k}\abs{\frac{(\mu -\hat{s}(r))}{1+C_r}}^pdr\right. \nonumber\\
& ~~+\sigma ^p C_p E\abs{\int _{0}^{t_1\wedge \tau _k}\frac{\hat{s}(r-\tau)^{2\gamma}\hat{s}(r)}{\left(1+C_r\right)^2}dr}^{\frac{p}{2}}\nonumber\\
&~~\left. +\delta ^pC_{p,\beta} E\abs{\int _{0}^{t_1\wedge \tau _k}\frac{\hat{s}(r)^p}{\left(1+C_r\right)^p}dr}\right]\nonumber\\
&\leq 4^{p-1}\left[E\abs{\xi(0)}^p+\lambda ^pT^{p-1}E\int _{0}^{t_1\wedge \tau _k}\abs{\mu -\hat{s}(r)}^pdr\right.\nonumber\\
&~~+\sigma ^p C_p E\abs{\int _{0}^{t_1\wedge \tau _k}\abs{\hat{s}(r-\tau)^{2\gamma}}\abs{\hat{s}(r)}dr}^{\frac{p}{2}}\nonumber\\
&~~\left.+\delta ^pC_{p,\beta} E\int _{0}^{t_1\wedge \tau _k}\abs{\hat{s}(r)}^pdr\right].
\end{align}
 Then, following the proof of  Lemma 3.1 in \cite{ji1}, 
 the proof of the stated result for $p >2$ is completed.\\
 For the case $0<p\leq 2$, the stated result follows directly from the H$\ddot{\text{o}}$lder inequality.
\end{proof}
So, Theorem \ref{ljuj1}, showed that  the $p$th moment  of the numerical solution of the balanced method (\ref{eju31}), is bounded  for any $p>0$.\\
Jiang et. al. \cite{ji1}, showed that Equation (\ref{eju1}) is mean reversion as $t\to \infty.$ Also, they proved that the Euler- 
Maruyama method keeps this property. In the following theorem we show that 
 $\mu$ is also an upper bound for the mean of solution of the BIM (\ref{eju31}), with step size $h<\frac{2}{\lambda}$, when $n\to \infty$.
\begin{thm}\label{ls1}
 For the BIM (\ref{eju31}) with control functions  (\ref{aju5}),  (\ref{eju30}) and (\ref{ajum5}), we have
\begin{equation}\label{s2}
E(s_{n+1})\leq (1-\lambda h)^n(E(\xi(0)) -\mu))+\mu +o(h^{\frac{1}{2}}),
\end{equation}
and hence for $h<\frac{2}{\lambda},$ we have $E(s_n)\leq \mu$ as $n\to\infty$.
\end{thm}
\begin{proof}
Taking expectation from the both sides of  the BIM (\ref{eju31}), and using $\frac{1}{1+C_n}\leq 1$,
   one  can derive that
\begin{align}\label{fi1}
E(s_{n+1})&=E(s_n)+E(\frac{\lambda \mu h}{1+C_n})-E(\frac{\lambda s_n\,h}{1+C_n})+\sigma E(\frac{s_{n-m}^\gamma \sqrt{s_n}\Delta W_n}{1+C_n})+\delta E(\frac{s_n \Delta \tilde {N}_n }{1+C_n})\nonumber\\
&\leq E(s_n)+\lambda \mu h-\lambda hE(\frac{ s_n}{1+C_n})+\sigma E(\frac{s_{n-m}^\gamma \sqrt{s_n}\Delta W_n}{1+C_n})+\delta E(\frac{s_n \Delta \tilde {N}_n }{1+C_n}).
\end{align}
 We then have
\begin{align}\label{fjui2}
-E(\frac{ s_n}{1+C_n})&=-E(s_n)+E(\frac{ s_nC_n}{1+C_n})\leq -E(s_n)+E(s_nC_n)\nonumber\\
&=-E(s_n)+C_0\,hE(s_n)+\sigma E(\sqrt{s_n}s_{n-m}^{\gamma}\abs{\Delta W_n}1_{s_n>\epsilon})\nonumber\\
&~~~+\sigma \epsilon ^{-\frac{1}{2}} E(s_ns_{n-m}^{\gamma}\abs{\Delta W_n}1_{s_n<\epsilon})+C_2\, E(s_n\abs{\Delta \tilde{N}_n}) .
\end{align}
For every $\gamma >0,$ from Theorem \ref{ljuj1} and the H$\ddot{\text{o}}$lder inequality,
there exists a constant $U_1>0$  such that
\begin{align}\label{fi4}
E(\sqrt{s_n}s_{n-m}^{\gamma})\leq E(s_n^3) ^{\frac{1}{6}}E(s_{n-m}^{\frac{6\gamma}{5}})
^{\frac{5}{6}}\leq U_1,
\end{align}
\begin{align}\label{fiju4}
E(s_ns_{n-m}^{\gamma})\leq E(s_n^6)^{\frac{1}{6}}E(s_{n-m}^{\frac{6\gamma}{5}})
^{\frac{5}{6}}\leq U_1,
\end{align}
\begin{align}\label{ju3}
E(s_n\abs{\Delta \tilde{N}_n})\leq E(s_n^2)^{\frac{1}{2}}E(\abs{\Delta \tilde{N}_n}^2)^{\frac{1}{2}}\leq U_1\,\sqrt{\beta h}.
\end{align}
We know that $E(\abs{\Delta W_n})=\sqrt{\frac{2h}{\pi}},$ also $s_n$ and 
$s_{n-m}$ are $\mathcal{F}_{t_n}-$measurable, so with substituting  inequalities
 (\ref{fi4}), (\ref{fiju4}) and (\ref{ju3}) in (\ref{fjui2}), we obtain
\begin{align}\label{fjui5}
-E(\frac{ s_n}{1+C_n})&\leq -E(s_n)+C_0\,hE(s_n)+\sigma \,\sqrt{\frac{2h}{\pi}} E(\sqrt{s_n}s_{n-m}^{\gamma}1_{s_n>\epsilon})\nonumber\\
&~~~+\sigma  \sqrt{\frac{2h}{\pi}}\epsilon ^{-\frac{1}{2}} E(s_ns_{n-m}^{\gamma}1_{s_n<\epsilon})+C_2 U_1\sqrt{\beta h}\nonumber\\
&\leq -E(s_n)+C_0\,hE(s_n)+U_1\, \sigma \,\sqrt{\frac{2h}{\pi}}(1+\epsilon ^{-\frac{1}{2}})+C_2 U_1\sqrt{\beta h}.
\end{align}
From the H$\ddot{\text{o}}$lder inequality and  $(\frac{1}{1+C_n})^2\leq 1$ and similar to the inequality  (\ref{fi4}), there exists a constant $U_2$  such that
\begin{align}\label{ju4}
E(\frac{s_{n-m}^\gamma \sqrt{s_n}}{1+C_n}\Delta W_n)&\leq E(s_{n-m}^{2\gamma} s_n)^{\frac{1}{2}}E((\frac{\Delta W_n}{1+C_n})^2)^{\frac{1}{2}}\leq E(s_{n-m}^{2\gamma} s_n)^{\frac{1}{2}}E(\Delta W_n^2)^{\frac{1}{2}}\leq U_2\,h^{\frac{1}{2}},
\end{align}
\begin{align}\label{ju5}
E(\frac{s_n \Delta \tilde {N}_n }{1+C_n})&\leq E(s_n^2)^{\frac{1}{2}}E((\frac{ \Delta \tilde {N}_n }{1+C_n})^2)^{\frac{1}{2}}\leq E(s_n^2)^{\frac{1}{2}}E(\Delta \tilde {N}_n ^2)^{\frac{1}{2}}\leq U_2 \sqrt{\beta h}.
\end{align}
Now, inequalities (\ref{fi1}), (\ref{fjui5}), (\ref{ju4}) and (\ref{ju5}) result 
\begin{align}\label{fi6}
E(s_{n+1})\leq E(s_n)(1-\lambda \,h)+\lambda \mu\,h+o(h^{\frac{1}{2}}).
\end{align}
This establishes the inequality (\ref{s2}).
\end{proof}
\section{Convergence analysis}
In this section, we prove the convergence of the BIM by using suitable stopping times and uniformly boundedness of the moments of $S(t)$ and $s(t).$\\
 For any integer $j,$ define the stopping times
\begin{equation*}
u_j:=\inf \{t\geq 0:\abs{S(t)}\geq j~\text{or}~S(t)<\frac{1}{j}\}
,~~~~~\nu _j:=\inf\{t \geq 0:\abs{s(t)}\geq j~\text{or}~s(t)<\frac{1}{j}\}
,~~~\rho _j:=u_j\wedge \nu _j,
\end{equation*}
and $\nu:=t\wedge\rho _j,$ for every $0\leq t\leq T$.
\begin{lem}\label{lem1}
For $h\in (0,1),$ there exist positive constants $M_1$ and $M_2$ such that 
\begin{align}\label{jjj1}
E\left(\int_{0}^{\nu}\frac{C_r}{1+C_r}dr\right)\leq M_1\,h^{\frac{1}{2}},
\end{align}
\begin{align}\label{j2}
E\left(\int_{0}^{\nu}\left(\frac{C_r}{1+C_r}\right)^2dr\right)\leq M_2\,h.
\end{align}
\end{lem}
\begin{proof}
 We  need the following version of (\ref{xi})
, i.e., there exists a positive constant $C_1,$ such that $E(\xi(r\wedge \rho _j-\tau)^{2\gamma})\leq C_1,$ for $0\leq r\wedge \rho _j <\tau$.\\
According to the defined control functions in (\ref{aju5}), (\ref{eju30}) and (\ref{ajum5}), for $h\in(0,1),$ and  $\frac{1}{1+C_r}\leq 1,$ we have\\
\begin{align}
E\left(\int_{0}^{\nu}\frac{C_r}{1+C_r}dr\right)&\leq E\int_{0}^{\nu}C_rdr \leq \int_{0}^{t}E\left(C_{r\wedge \rho _j}\right)dr\leq \int_{0}^{T} E(C_{r\wedge \rho _j})dr \nonumber
\end{align}
\begin{align} \label{j1}
&=\int_{0}^{T} E\left(C_{0}\,h+\sigma\frac{\hat {s}(r\wedge \rho _j -\tau)^{\gamma}}{\sqrt{\hat{s}(r\wedge \rho _j)}}\abs{ \Delta{W}(r\wedge \rho _j)}1_{\hat{s}(r\wedge \rho _j) \geq \epsilon}\right. \nonumber\\
&\left. ~~~+\sigma\frac{\hat {s}(r\wedge \rho _j -\tau)^{\gamma}}{\sqrt{\epsilon}}\abs{ \Delta{W}(r\wedge \rho _j)}1_{\hat{s}(r\wedge \rho _j) < \epsilon}+C_2\abs{\Delta{\tilde {N}}(r\wedge \rho _j)}\right)dr \nonumber\\
&\leq\int_{0}^{T} E\left(C_0\,h+\sigma j^{\gamma +\frac{1}{2}}\,\abs{ \Delta{W}(r\wedge \rho _j)}1_{{\{\hat{s}(r\wedge \rho _j) \geq \epsilon}\wedge (r\wedge \rho _j \geq \tau)\}}\right.
\nonumber\\
& ~~~+\sigma j^{ \frac{1}{2}}\xi(r\wedge \rho _j-\tau)^{\gamma}\,\abs{ \Delta{W}(r\wedge \rho _j)}1_{{\{\hat{s}(r\wedge \rho _j) \geq \epsilon}\wedge (0\leq r\wedge \rho _j < \tau)\}}\nonumber\\
&~~~+    \sigma j^{\gamma}\,\epsilon ^{-\frac{1}{2}}\,\abs{ \Delta{W}(r\wedge \rho _j)}1_{{\{\hat{s}(r\wedge \rho _j) < \epsilon}\wedge (r\wedge \rho _j \geq \tau)\}})\nonumber\\
& ~~~+ \sigma \epsilon ^{-\frac{1}{2}}\xi(r\wedge \rho _j-\tau)^{\gamma}\abs{ \Delta{W}(r\wedge \rho _j)}1_{{\{\hat{s}(r\wedge \rho _j) < \epsilon}\wedge (0 \leq r\wedge \rho _j < \tau)\}}
\nonumber\\
& ~~~\left. +C_2\abs{\Delta{\tilde {N}}(r\wedge \rho _j)}\right).
\end{align}
The Cauchy Schwarz  inequality  implies
\begin{align}\label{je1}
E(\xi(r\wedge \rho _j-\tau)^{\gamma}\abs{ \Delta{W}(r\wedge \rho _j)})\leq E(\xi(r\wedge \rho _j-\tau)^{2\gamma})^{\frac{1}{2}}\,E(\abs{ \Delta{W}(r\wedge \rho _j)}^2)^{\frac{1}{2}}\leq   C_1^{\frac{1}{2}}\,h^{\frac{1}{2}}.
\end{align}
Then, using (\ref{je1}) in (\ref{j1}), we obtain
\begin{align*}
E\int_{0}^{\nu}\frac{C_r}{1+C_r}dr\leq T(C_0\,h+\sigma (j^{\gamma +\frac{1}{2}}+C_1^{\frac{1}{2}}\,j^{\frac{1}{2}})\,h^{\frac{1}{2}}+\sigma (C_1^{\frac{1}{2}}+j^{\gamma})\,\epsilon ^{-\frac{1}{2}}\,h^{\frac{1}{2}}+C_2 \sqrt{\beta} h^{\frac{1}{2}})=:M_1\,h^{\frac{1}{2}}.
\end{align*}
Moreover, similarly, there exists a constant $M_2,$ such that
\begin{align*}
E\int_{0}^{\nu}\left(\frac{C_r}{1+C_r}\right)^2dr\leq M_2\,h.
\end{align*}
\end{proof}
\begin{lem}\label{lj2}
Let $S(t)$ be the solution of equation (\ref{eju1}).
Then
\begin{equation}\label{ej1}
\lim_{h\to 0}(\sup_{0\leq t\leq T}E\abs{S(t\wedge \rho_j)-s(t\wedge\rho _j)})=0.
\end{equation}
\end{lem}
\begin{proof}
Let $a_0=1$ and $a_n=exp(-\frac{n(n+1)}{2})$ for $n\geq 1$, so that $\int_{a_n}^{a_{n-1}}\frac{du}{u}=n.$ For each $n\geq 1,$ there exists a continuous function $\psi _n(u)$ with support in $(a_n,a_{n-1}),$ such that
\begin{equation*}
0\leq \psi _n(u)\leq \frac{2}{nu}~~~\text{for}~~a_n<u<a_{n-1}
\end{equation*}
and $\int _{a_n}^{a_{n-1}}\psi _n(u)=1.$ Define
\begin{equation}\label{ej2}
\phi _n(x)=\int _{0}^{\abs{x}}dy\int _{0}^{y}\psi _n(u)du.
\end{equation}
Then $\phi _n\in C^2(\mathbb{R},\mathbb{R})$, $\phi _n(0)=0,$ and for all $x\in \mathbb{R},$ $\abs{\phi ^{\prime}_n(x)}\leq 1$, also for every $a_n<\abs{x}<a_{n-1},$\\
 $\abs{\phi ^{\prime \prime}_n(x)}\leq \frac{2}{n\abs{x}}$ and otherwise $\abs{\phi ^{\prime \prime}_n(x)}=0.$
One can  easily observe that
\begin{equation}\label{ej5}
\abs{x}-a_{n-1}\leq \phi_n(x)\leq \abs{x}.
\end{equation}
Let  $e(\nu):=S(\nu)-s(\nu).$ From (\ref{ej5}), we have
\begin{equation}\label{ejj50}
E(\abs{e(\nu)})\leq a_{n-1}+E(\phi _n(e(\nu))).
\end{equation}
 Applying the It$\hat{\text{o}}$'s formula, using  $\frac{1}{1+C_r}\leq 1$ and definition of $\phi _n$, we  derive
\begin{align}\label{ejjju5}
E(\phi _n(e(\nu))&=\lambda \,\mu E\int_{0}^{\nu}\phi _n ^{\prime }(e(r))(\frac{C_r}{1+C_r})dr-(\lambda +\delta \beta) E\int_{0}^{\nu}\phi _n^{\prime }(e(r))(S(r)- \frac{\hat{s}(r)}{1+C_r})dr\nonumber\\
&~~~+\frac{\sigma ^2}{2}E\int_{0}^{\nu}\phi _n ^{\prime \prime}(e(r))\left[S^{\gamma}(r-\tau)\sqrt{S(r)}-\frac{\hat{s}^{\gamma}(r-\tau)\sqrt{\hat{s}(r)}}{1+C_r}\right]^2dr\nonumber\\
&~~~+\beta E\int_{0}^{\nu}[\phi _n((1+\delta)e(r))-\phi _n(e(r))]dr\nonumber\\
&
\leq \lambda \,\mu E\int_{0}^{\nu}\frac{C_r}{1+C_r}dr+(\lambda +\delta \beta) E\int_{0}^{\nu}\abs{S(r)- \frac{\hat{s}(r)}{1+C_r}}dr\nonumber\\
&~~~+\frac{\sigma ^2}{2}E\int_{0}^{\nu}\phi _n^{\prime \prime}(e(r))\left[S^{\gamma}(r-\tau)\sqrt{S(r)}-\frac{S^{\gamma}(r-\tau)\sqrt{\hat{s}(r)}}{1+C_r}\right. \nonumber\\
&\left. ~~~ +\frac{S^{\gamma}(r-\tau)\sqrt{\hat{s}(r)}}{1+C_r}-\frac{\hat{s}^{\gamma}(r-\tau)\sqrt{\hat{s}(r)}}{1+C_r}\right]^2dr\nonumber\\
&~~~+ \delta \beta E\left( \sup _{x\in \mathbb{R}}\abs{ \phi _n^{ \prime} (x)} \int_{0}^{\nu}\abs{e(r)}dr\right)\nonumber\\
&\leq \lambda \,\mu E\int_{0}^{\nu}\frac{C_r}{1+C_r}dr+(\lambda +\delta \beta) E\int_{0}^{\nu}\left(\abs{S(r)- \frac{S(r)}{1+C_r}}+\abs{ \frac{S(r)}{1+C_r}- \frac{\hat{s}(r)}{1+C_r}}\right)dr\nonumber\\
&~~~+\sigma ^2 E\int_{0}^{\nu}\phi _n^{\prime \prime}(e(r))\left[S^{\gamma}(r-\tau)\sqrt{S(r)}-\frac{S^{\gamma}(r-\tau)\sqrt{\hat{s}(r)}}{1+C_r}\right]^2dr\nonumber\\
&~~~+\sigma ^2 E\int_{0}^{\nu}\phi _n^{\prime \prime}(e(r))\left[\frac{S^{\gamma}(r-\tau)\sqrt{\hat{s}(r)}}{1+C_r}-\frac{\hat{s}^{\gamma}(r-\tau)\sqrt{\hat{s}(r)}}{1+C_r}\right]^2dr\nonumber\\
&~~~+\delta \beta E\int_{0}^{\nu}\abs{e(r)}dr
 \nonumber\\
&\leq (\lambda \mu+\lambda \,j+\delta \beta \,j)  E\int_{0}^{\nu}\frac{C_r}{1+C_r}dr\nonumber\\
&~~~+(\lambda +2\delta \beta)  E\int_{0}^{\nu} \abs{e(r)}+(\lambda +\delta \beta)  E\int_{0}^{\nu}\abs{ s(r)- \hat{s}(r)}dr\nonumber\\
&~~~+\sigma ^2 j^{2\gamma}E\int_{0}^{\nu}\abs{\phi _n^{\prime \prime}(e(r))}\left[\sqrt{S(r)}-\frac{\sqrt{\hat{s}(r)}}{1+C_r}\right]^2dr\nonumber\\
&~~~+\sigma ^2 j E\int_{0}^{\nu}\abs{\phi _n^{\prime \prime}(e(r))}\left[S^{\gamma}(r-\tau)-\hat{s}^{\gamma}(r-\tau)\right]^2dr.
\end{align}
Now, in order to bound the right-hand side of (\ref{ejjju5}), for simplicity, we assign each term to $J_1,J_2,J_3,J_4,$ $J_5,$ respectively.\\
To bound the term $J_3,$ by definition (\ref{eju6j}), for $r\in[0,\nu]$, we have
\begin{align}
s(r)-\hat{s}(r)=\frac{\lambda (\mu -s_{[\frac{r}{h}]})(r-[\frac{r}{h}]h)}{1+C_{[\frac{r}{h}]}}+\frac{\sigma s^{\gamma}_{{([\frac{r}{h}]-N)}}\sqrt{s_{[\frac{r}{h}]}}(W(r)-W([\frac{r}{h}]h))}{1+C_{[\frac{r}{h}]}}+\frac{\delta s_{[\frac{r}{h}]}(\tilde {N}(r)-\tilde {N}([\frac{r}{h}]h))}{1+C_{[\frac{r}{h}]}} .
\end{align}
So, for every $h\in(0,1)$,
\begin{align}\label{ej8}
E\int_{0}^{\nu}\abs{s(r)-\hat{s}(r)}dr&\leq \lambda (\mu +j)\,h\,T+\sigma j^{\gamma+\frac{1}{2}}E\int_{0}^{\nu}\abs{W(r)-W([\frac{r}{h}]h)}dr\nonumber\\
&~~~+\delta j E\int_{0}^{\nu}\abs{\tilde {N}(r)-\tilde {N}([\frac{r}{h}]h)}dr
\nonumber\\
&\leq \lambda (\mu +j)\,h\,T+\sigma j^{\gamma+\frac{1}{2}}\int_{0}^{T}E\abs{W(r\wedge \rho _j)-W([\frac{r\wedge \rho _j}{h}]h)}dr\nonumber\\
&~~~+\delta j E\int_{0}^{T}\abs{\tilde {N}(r\wedge \rho _j)-\tilde {N}([\frac{r\wedge \rho _j}{h}]h)}dr\nonumber\\
&\leq \lambda (\mu +j)\,h\,T+ \sigma T j^{\gamma+\frac{1}{2}}h^{\frac{1}{2}}+\delta jT\sqrt{\beta}h^{\frac{1}{2}} =:D h^{\frac{1}{2}}.
\end{align}
To bound $J_4$, from definition of $\phi _n$ and using Lemma \ref{lem1} and inequality (\ref{ej8}), we obtain
\begin{align}\label{ej7}
E\int_{0}^{\nu}&\abs{\phi _n^{\prime \prime}(e(r))}\left[\sqrt{S(r)}-\frac{\sqrt{\hat{s}(r)}}{1+C_r}\right]^2dr=E\int_{0}^{\nu}\abs{\phi _n^{\prime \prime}(e(r))}\left[\sqrt{S(r)}-\sqrt{\hat{s}(r)}+\sqrt{\hat{s}(r)}-\frac{\sqrt{\hat{s}(r)}}{1+C_r}\right]^2dr\nonumber\\
&\leq 2 E\int_{0}^{\nu}\abs{\phi _n^{\prime \prime}(e(r))}\left[\sqrt{S(r)}-\sqrt{\hat{s}(r)}\right]^2dr+ 2E\int_{0}^{\nu}\abs{\phi _n^{\prime \prime}(e(r))}\left[\sqrt{\hat{s}(r)}-\frac{\sqrt{\hat{s}(r)}}{1+C_r}\right]^2dr\nonumber\\
~~~& \leq 2 E\int_{0}^{\nu}\abs{\phi _n^{\prime \prime}(e(r))}\abs{S(r)-\hat{s}(r)}dr +2jE\int_{0}^{\nu}\abs{\phi _n^{\prime \prime}(e(r))}\left(\frac{C_r}{1+C_r}\right)^2dr\nonumber\\
&\leq 2 E\int_{0}^{\nu}\abs{\phi _n^{\prime \prime}(e(r))}\abs{S(r)-\hat{s}(r)}dr + \frac{4j}{na_n}E\int_{0}^{\nu}\left(\frac{C_r}{1+C_r}\right)^2dr\nonumber\\
&\leq 2E\int_{0}^{\nu}\frac{2}{n}dr+\frac{4}{na_n}E\int_{0}^{\nu}\abs{s(r)-\hat{s}(r)}dr+\frac{4j}{na_n}\,M_2\,h\nonumber\\
&\leq \frac{4T}{n}+\frac{4j}{na_n}M_2\,h+\frac{4}{na_n}D h^{\frac{1}{2}}.
\end{align}
Now considering $J_5$ and definition of $\phi _n$,
\begin{align}\label{ej6}
E\int_{0}^{\nu}&\abs{\phi _n^{\prime \prime}(e(r))}\left[S^{\gamma}(r-\tau)-\hat{s}^{\gamma}(r-\tau)\right]^2dr\nonumber\\
&\leq C_j E\int_{0}^{\nu}\abs{\phi _n^{\prime \prime}(e(r))}\left[S(r-\tau)-\hat{s}(r-\tau)\right]^{2\gamma}dr\nonumber\\
&\leq \frac{2C_j\bar{C}}{na_n}E\int_{0}^{\nu}\left[s(r-\tau)-\hat{s}(r-\tau)\right]^{2\gamma}dr
+\frac{2C_j\bar{C}}{na_n}E\int_{0}^{\nu}\abs{e(r-\tau)}^{2\gamma}dr,
\end{align}
where
\begin{equation*}
\bar{C}=\left\{\begin{array}{l}
1,~~~~~~0<\gamma\leq\frac{1}{2},\\
2^{2\gamma -1},~~~~~~\gamma >\frac{1}{2},
\end{array}\right.
\end{equation*}
and
\begin{equation*}
C_j=\left\{\begin{array}{l}
\bar{C}^2_j,~~~~~~\gamma >1,\\
1,~~~~~~\text{otherwise},
\end{array}\right.
\end{equation*}
The last inequality is true due to the following fact. For $x,y>0$  and $\theta \in(0,1],$
\begin{equation*}
\abs{x^{\theta}-y^{\theta}}\leq \abs{x-y}^{\theta},
\end{equation*}
 and  for $\abs{x}\leq j,$ $\abs{y}<j$ and $\theta >1,$
\begin{equation*}
\abs{x^{\theta}-y^{\theta}}\leq \bar{C}_j\abs{x-y}^{\theta},
\end{equation*}
 Here $\bar{C}_j,$ is a constant depending on $j.$
  Substituting
(\ref{jjj1}), (\ref{ej8}), (\ref{ej7}) and (\ref{ej6}),   in
(\ref{ejjju5}), we derive
\begin{align*}
E(\phi _n(e(\nu)))&\leq \sigma ^2 j \left(\frac{2C_j\bar{C}}{na_n}E\int_{0}^{\nu}\left[s(r-\tau)-\hat{s}(r-\tau)\right]^{2\gamma}dr
+\frac{2C_j\bar{C}}{na_n}E\int_{0}^{\nu}\abs{e(r-\tau)}^{2\gamma}dr\right)\nonumber\\
&~~~+\sigma ^2 j ^{2\gamma} (\frac{4T}{n}+\frac{4j}{na_n}M_2\,h)+(\lambda +2 \delta \beta) E\int_{0}^{\nu}\abs{e(r)}dr\nonumber\\
&~~~+(\lambda \mu +\lambda j+ \delta \beta j)M_1h^{\frac{1}{2}}+(\lambda + \delta \beta) D\,h^{\frac{1}{2}}+D\sigma ^2 j^{2\gamma}\frac{4}{na_n}\,h^{\frac{1}{2}} .
\end{align*}
 From (\ref{ejj50}), we then obtain
\begin{align}
E\abs{e(\nu)}&\leq a_{n-1}+\sigma ^2 j ^{2\gamma} (\frac{4T}{n}+\frac{4j}{na_n}M_2\,h)+(\lambda \mu +\lambda j+ \delta \beta j)M_1\,h^{\frac{1}{2}}\nonumber\\
&~~~+(\lambda + \delta \beta)D\,h^{\frac{1}{2}}+D\sigma ^2 j^{2\gamma}\frac{4}{na_n}\,h^{\frac{1}{2}}+(\lambda +2 \delta \beta) E\int_{0}^{\nu}\abs{e(r)}dr\nonumber\\
&~~~+\sigma ^2 j \left(\frac{2C_j\bar{C}}{na_n}E\int_{0}^{\nu}\left[s(r-\tau)-\hat{s}(r-\tau)\right]^{2\gamma}dr
+\frac{2C_j\bar{C}}{na_n}E\int_{0}^{\nu}\abs{e(r-\tau)}^{2\gamma}dr\right)\nonumber\\
&=:a_{n-1}+\frac{\alpha _1}{n}+\alpha _2 h^{\frac{1}{2}}+\alpha _3 E\int_{0}^{\nu}\left[s(r-\tau)-\hat{s}(r-\tau)\right]^{2\gamma}dr\nonumber\\
&~~~+\frac{\alpha _4 }{na_n}E\int_{0}^{\nu}\abs{e(r-\tau)}^{2\gamma}dr+(\lambda +2 \delta \beta) E\int_{0}^{\nu}\abs{e(r)}dr,
\end{align}
where $\alpha _1$ and $\alpha _4$ are independent of $n,$ and
$\alpha _2,$ $\alpha _3$ depend on $n.$ Then, following the  proof of Lemma 4.2 in \cite{ji1} (right after Equation (22)),  this
lemma is proved.
\end{proof}
\begin{lem}\label{ljju3}
For the stopping times $u_j,$ $\nu _j$ and $\rho _j$, we have
\begin{equation}\label{ej19}
\lim _{h\to 0}E\left(\sup_{0\leq t\leq T} \abs{S(t\wedge \rho_j)-s(t\wedge \rho_j)}^2\right)=0.
\end{equation}
\end{lem}
\begin{proof}
From (\ref{eju6j}) and the
H$\ddot{\text{o}}$lder inequality, we get
\begin{align}\label{aju4}
(S(\nu )-s(\nu))^2&\leq 4T\lambda ^2 \mu ^2 \int _{0}^{\nu} \left(1-\frac{1}{1+C_r}\right)^2dr+4T\lambda ^2 \int _{0}^{\nu}\left(S(r)-\frac{\hat {s}(r)}{1+C_r}\right)^2dr \nonumber\\
&~~~+4\sigma ^2 \left[\int _{0}^{\nu}\left(S(r-\tau)^{\gamma}\sqrt{S(r)}-\frac{\hat{s}^{\gamma}(r-\tau)\sqrt{\hat{s}(r)}}{1+C_r}\right)dW(r)\right]^2\nonumber\\
&~~~+4\delta ^2 \left[\int _{0}^{\nu}\left(S(r)-\frac{\hat {s}(r)}{1+C_r}\right)d\tilde {N}(r)\right]^2.
\end{align}
Let $\nu _1=t_1\wedge \rho _j.$ By the Doob martingale inequality
\cite{r16},
 for every $t_1\in[0,T],$ we have
\begin{align}\label{ejju10}
&E\left(\sup_{0\leq t\leq t_1} \abs{S(t\wedge \rho_j)-s(t\wedge \rho_j)}^2\right)\nonumber\\
&~~~~\leq 4T\lambda ^2 \mu ^2 E \int _{0}^{\nu _1} \left(1-\frac{1}{1+C_r}\right)^2dr+(4T\lambda ^2+16\beta \delta ^2)E \int _{0}^{\nu _1}\left(S(r)-\frac{\hat {s}(r)}{1+C_r}\right)^2dr \nonumber\\
&~~~~~~~+ 16\sigma ^2 E\int _{0}^{\nu _1}\left[S(r-\tau)^{\gamma}\sqrt{S(r)}-\frac{\hat{s}^{\gamma}(r-\tau)\sqrt{\hat{s}(r)}}{1+C_r}\right]^2dr.
\end{align}
On the other hand, from  (\ref{j2}),  $\frac{1}{1+C_r}\leq 1$ and
(\ref{ej8}), we have
\begin{align}\label{a2}
&E \int _{0}^{\nu _1}\left(S(r)-\frac{\hat {s}(r)}{1+C_r}\right)^2dr \leq 2 E \int _{0}^{\nu _1}\left(S(r)-\frac{S(r)}{1+C_r}\right)^2dr+2E\int _{0}^{\nu _1}\left(\frac{S(r)}{1+C_r}-\frac{\hat {s}(r)}{1+C_r}\right)^2dr\nonumber\\
&~~~~~~~~~\leq 2j^2 E \int _{0}^{\nu _1}\left(\frac{C_r}{1+C_r}\right)^2dr+4E\int _{0}^{\nu _1}(S(r)-s(r))^2dr+4E\int _{0}^{\nu _1}(s(r)-\hat{s}(r))^2dr\nonumber\\
&~~~~~~~~~\leq 2j^2M_2\,h+8jE\int _{0}^{\nu _1}\abs{s(r)-\hat{s}(r)}dr+4E\int _{0}^{\nu _1}(S(r)-s(r))^2dr\nonumber\\
&~~~~~~~~~\leq 2j^2M_2\,h+8 j\,D\,h^{\frac{1}{2}}+4E\int _{0}^{\nu _1}(S(r)-s(r))^2dr.
\end{align}
Similarly, from (\ref{ej8}) and the inequality $\frac{1}{1+C_r}\leq 1,$ we also
get
\begin{align}
&E\int _{0}^{\nu _1}\left[S(r-\tau)^{\gamma}\sqrt{S(r)}-\frac{\hat{s}^{\gamma}(r-\tau)\sqrt{\hat{s}(r)}}{1+C_r}\right]^2dr\nonumber\\
&~~~~=
E\int_{0}^{\nu _1}\left[S^{\gamma}(r-\tau)\sqrt{S(r)}-\frac{S^{\gamma}(r-\tau)\sqrt{\hat{s}(r)}}{1+C_r} +\frac{S^{\gamma}(r-\tau)\sqrt{\hat{s}(r)}}{1+C_r}-\frac{\hat{s}^{\gamma}(r-\tau)\sqrt{\hat{s}(r)}}{1+C_r}\right]^2dr\nonumber\\
&~~~~\leq 2 E\int_{0}^{\nu _1}\left[S^{\gamma}(r-\tau)\sqrt{S(r)}-\frac{S^{\gamma}(r-\tau)\sqrt{\hat{s}(r)}}{1+C_r}\right]^2dr\nonumber\\
&~~~~~~~+2E\int_{0}^{\nu _1}\left[\frac{S^{\gamma}(r-\tau)\sqrt{\hat{s}(r)}}{1+C_r}-\frac{\hat{s}^{\gamma}(r-\tau)\sqrt{\hat{s}(r)}}{1+C_r}\right]^2dr\nonumber\\
&~~~~\leq 2 j^{2\gamma} E\int_{0}^{\nu _1}\left[\sqrt{S(r)}-\frac{\sqrt{\hat{s}(r)}}{1+C_r}\right]^2dr+2 jE\int_{0}^{\nu _1}(S^{\gamma}(r-\tau)-\hat{s}(r-\tau)^{\gamma})^2dr\nonumber
\end{align}
\begin{align}\label{a1}
&~~~~\leq 2 j^{2\gamma}E\int_{0}^{\nu _1}\left[\sqrt{S(r)}-\sqrt{\hat{s}(r)}+\sqrt{\hat{s}(r)}-\frac{\sqrt{\hat{s}(r)}}{1+C_r}\right]^2dr\nonumber\\
&~~~~~~~+ 2j C_j E\int_{0}^{\nu _1} \abs{S(r-\tau)-\hat{s}(r-\tau)}^{2\gamma}dr \nonumber\\
&~~~~\leq 4j^{2\gamma}E\int_{0}^{\nu _1}\abs{S(r)-\hat{s}(r)}dr+ 4 j^{2\gamma +1}M_2\,h+2 j C_j E\int_{0}^{\nu _1} \abs{S(r)-\hat{s}(r)}^{2\gamma}dr\nonumber\\
& ~~~~\leq 4j^{2\gamma}E\int_{0}^{\nu _1}\abs{S(r)-s(r)}dr +4D\, j^{2\gamma}\,h^{\frac{1}{2}}+4 j^{2\gamma +1}M_2\,h+ 2j C_j E\int_{0}^{\nu _1} \abs{S(r)-\hat{s}(r)}^{2\gamma}dr.
\end{align}
Then, applying (\ref{ejju10}), (\ref{a2}), (\ref{a1}) and following the proof of Lemma 4.3 in \cite{ji1} (right after Equation (40)),  the proof is
completed.
\end{proof}
\begin{thm}\label{tj1}
Let $S(t)$ be the solution of the equation (\ref{eju1}). Then the numerical solution (\ref{eju6j}) converges to $S(t)$, that is,
\begin{equation}\label{ejju24}
\lim _{h\to 0}E\left(\sup_{0\leq t\leq T} \abs{S(t)-s(t)}^2\right)=0.
\end{equation}
\end{thm}
\begin{proof}
By Theorem \ref{ljuj1} and Lemma \ref{ljju3} and in the same way as Theorem 4.1 in \cite{ji1} the conclusion follows.
\end{proof}
\section{Numerical examples}
In this section, we illustrate some numerical examples that confirm the results  in the previous sections. Also, by the convergence theory in Section 4, we show that the BIM can be used to compute some financial quantities.\\
Consider the  delay CIR model with jump
\begin{equation}\label{exj1}
\left\{\begin{array}{l}
dS(t)=\lambda (\mu -S(t))dt+\sigma S(t-1)^{\gamma}\sqrt{S(t)}dW(t)+\delta S(t)d\tilde{N}(t),~~t\geq 0,\\
S(t)=1,~~t\in [-1,0],
\end{array}\right.
\end{equation}
We consider the two following examples.\\
 \noindent {\bf{Example 1.} }$\lambda=5, \mu=0.5, \sigma=1.5, \gamma=0.5, \delta=1, \beta=2.$\\
  \noindent {\bf{Example 2.} }$\lambda=100, \mu=5, \sigma=2, \gamma=1, \delta=2, \beta=4.$\\
  Figs. \ref{Fj1}-\ref{Fj4} show the values $S(t)$ vs. $t$ for Examples 1 and 2 by the balanced and Euler methods, with ten solution paths. From these figures it can be observed that the balanced method preserves non-negativity of the solution of these Examples even for the large step size $h=0.5$, while the Euler method does not preserve this property for these Examples even  for the small step size $h=0.01$. 
In Figs. \ref{Fj5}, \ref{Fj6}, we apply the BIM to Examples 1 and 2.  We estimates the rate of convergence by drawing the strong error at the endpoint $T=1,$ $e_h^{strong}:=E\abs{S(T)-s_T}^2$. We plot $e_h^{strong}$ against $h$ on a log-log scale.  Since we do not have an explicit solution of Examples 1 and 2, we take the BIM  with  step size $h=2^{-11}$ as a reference solution. For showing the convergence, we compare the reference solution with  the BIM  evaluated with $2^{2i-1}\,h,~i=1,2,3,4,5.$ We  compute $500$ different solution paths. Also,  we apply  the Euler method for Examples 1 and 2, for comparison purpose. From these figures it can be seen that, the rate of convergence of the balanced method is better than the Euler method for both of  Examples. Figs. \ref{Fsj1}- \ref{Fsj4}, show the values of $E(S(t))$ and $E(S(t)^2)$ vs. $t$ for the Examples 1 and 2, by the BIM with step size $h=0.1$ and with 1000 solution paths. Figs. \ref{Fsj1} and \ref{Fsj2}, show $\lim _{n\to\infty}E(s_n)\leq \mu$; similarly, Figs. \ref{Fsj3} and \ref{Fsj4}, show $E(S(t)^2)$ is bounded, confirming the results of Theorems \ref{ljuj1}, \ref{ls1}.
\\
Now, we use bonds and barrier options to show our results.\\
 \noindent {\bf{Example 3.}  }(Bonds): In the case where  the SDDE with jump (\ref{eju1}) describes short-term interest rate dynamics, the price of a bond at the end of period is given by
\begin{align*}
B(T):=E\left[exp\left(-\int_{0}^{T}S(t)dt\right)\right].
\end{align*}
By the step process $\hat{s}(t)$ in (\ref{e5}), a natural approximation to compute $B(T)$ is 
\begin{align*}
\hat{B}_h(T):=E\left[exp\left(-\int_{0}^{T}\hat{s}(t)dt\right)\right].
\end{align*}
We have
\begin{align*}
\lim_{h\to 0}|B(T)-\hat{B}_h(T)|=0.
\end{align*}
The proof is the similar to that of Theorem 4.1 in \cite{hi1}.\\
 \noindent {\bf{Example 4.}  } (A path dependent option): Let $S(t)$ be the solution of the equation (\ref{eju1}) and $\hat{s}(t)$ be the BIM process defined by (\ref{e5}). We consider an up-and-out call option with the expiry time $T,$ the  exercise price $K$ and the fixed barrier $B$. Payoff of this option  at expiry time $T$ is $(S(T)-K)^{+},$ if $S(t)$ never decreases below the fixed barrier $B$ and is zero otherwise. We suppose that the expected payoff is computed from (\ref{e5}). Define
\begin{align*}
V:=E[(S(T)-K)^{+}1_{0\leq S(t) \leq B,0\leq t\leq T}];
\end{align*}
and
\begin{align*}
\hat{V}_h:=E[(\hat{s}(T)-K)^{+}1_{0\leq \hat{s}(t) \leq B,0\leq t\leq T}];
\end{align*}
where $K$ and $B$ are constants. We have
\begin{align*}
\lim _{h\to 0}|V-\hat{V}_h|=0.
\end{align*}
The proof is  the same to that of Theorem 5.1 in \cite{hi1}.
\begin{figure}
\begin{center}
\includegraphics[width=12.cm]{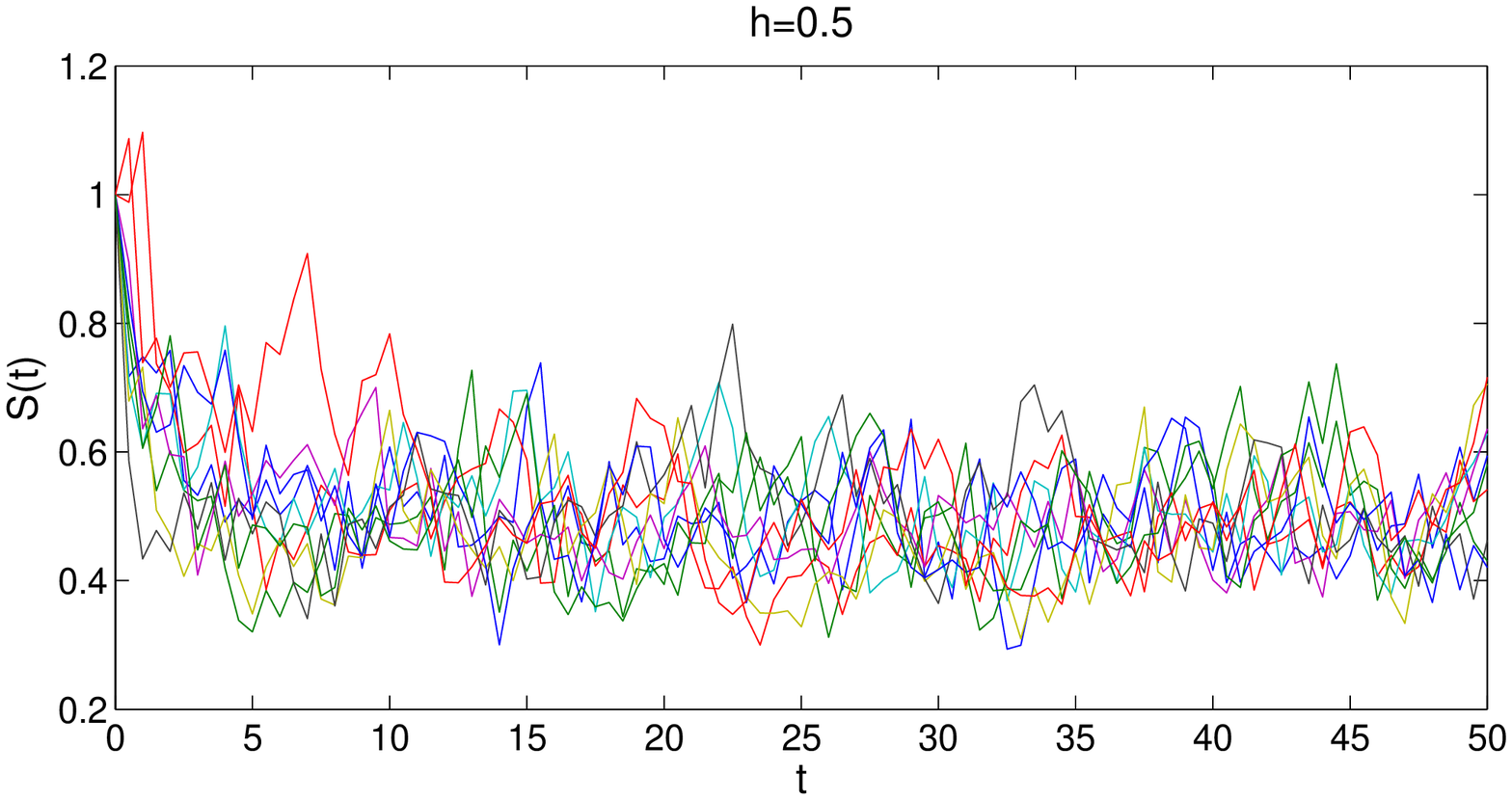}\\
\includegraphics[width=12cm]{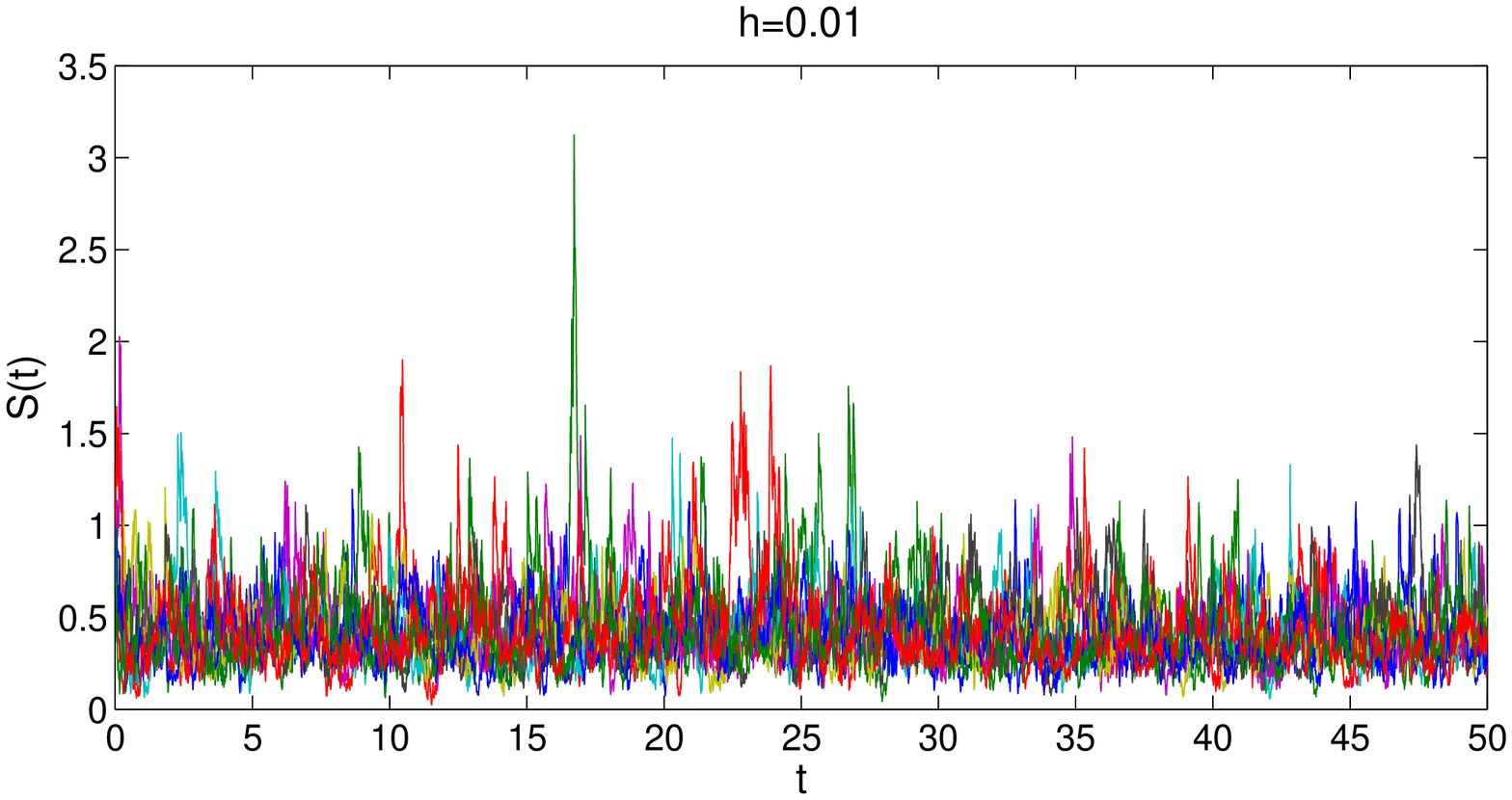}
\caption{Ten solution paths of
 Example 1, approximated by BIM  with $C_0=10,\,C_2=1$ and $C_1=\sigma \frac{s_{n-m}^{\gamma}}{\sqrt{s_n}}.$ }
\label{Fj1}
\includegraphics[width=12cm]{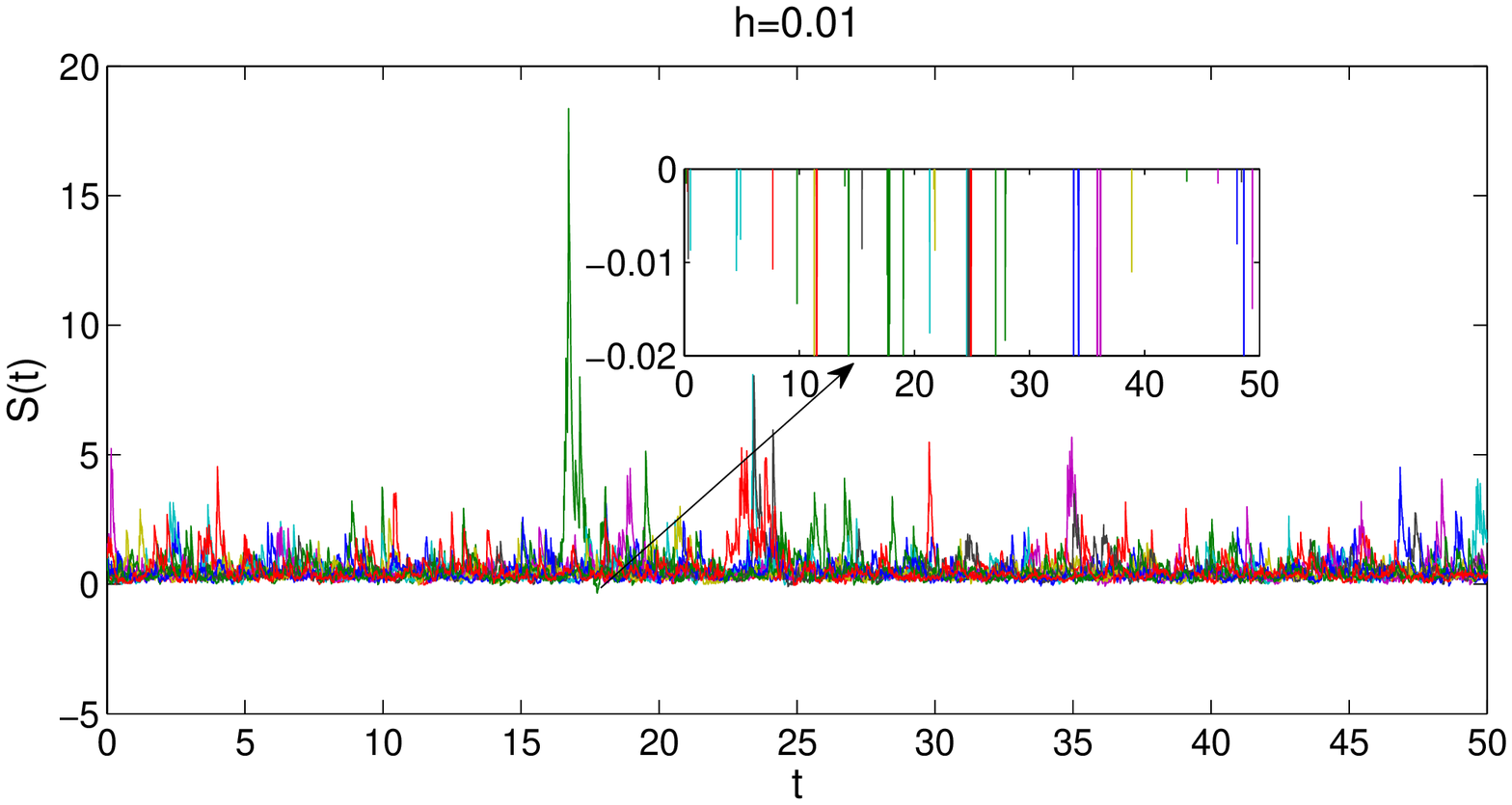}
\caption{Ten solution paths of
 Example 1,  approximated by Euler method. }
 \label{Fj2}
 \end{center}
\end{figure}
 \begin{figure}
\begin{center}
\includegraphics[width=12.cm]{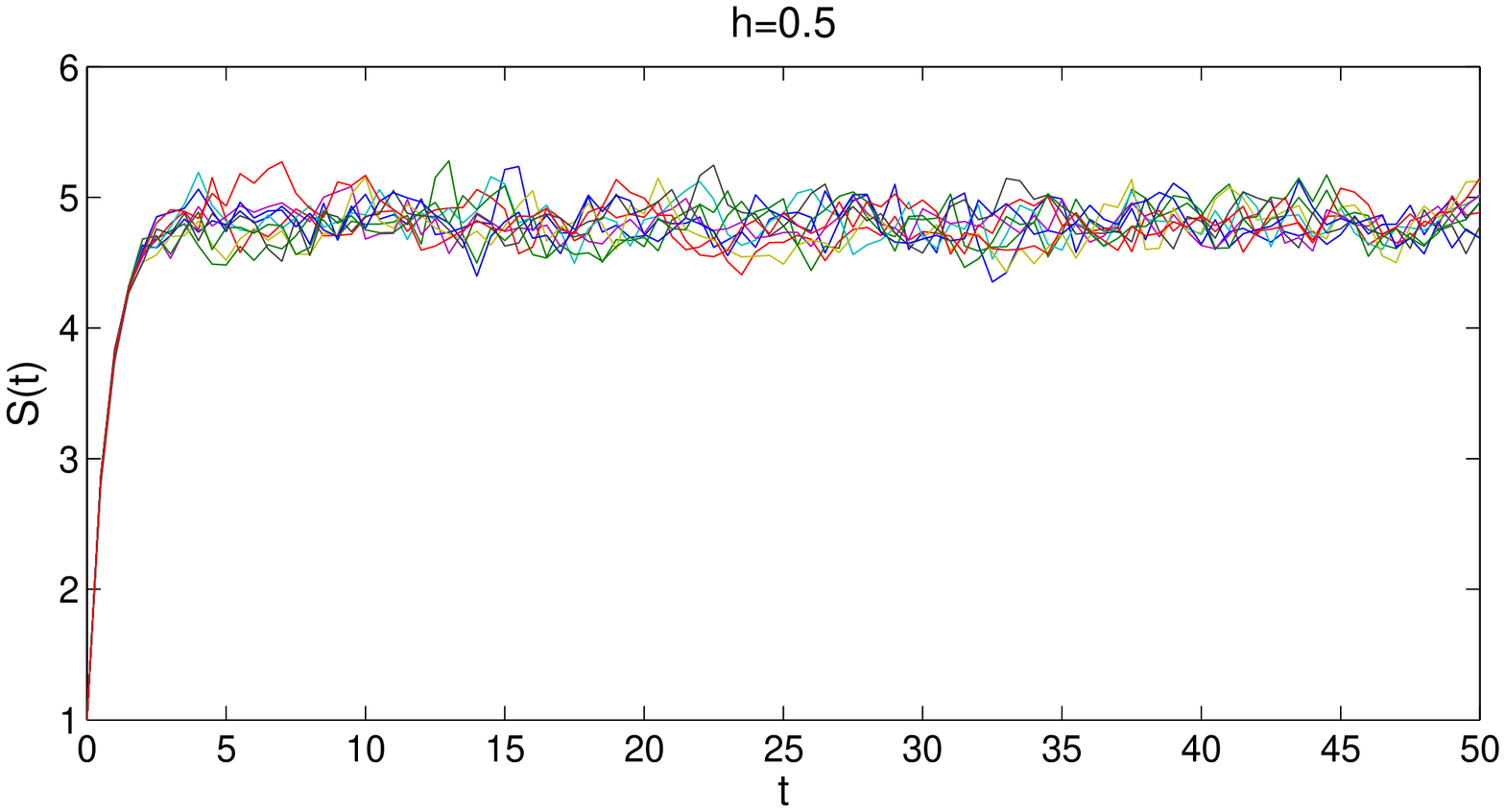}\\
\includegraphics[width=12cm]{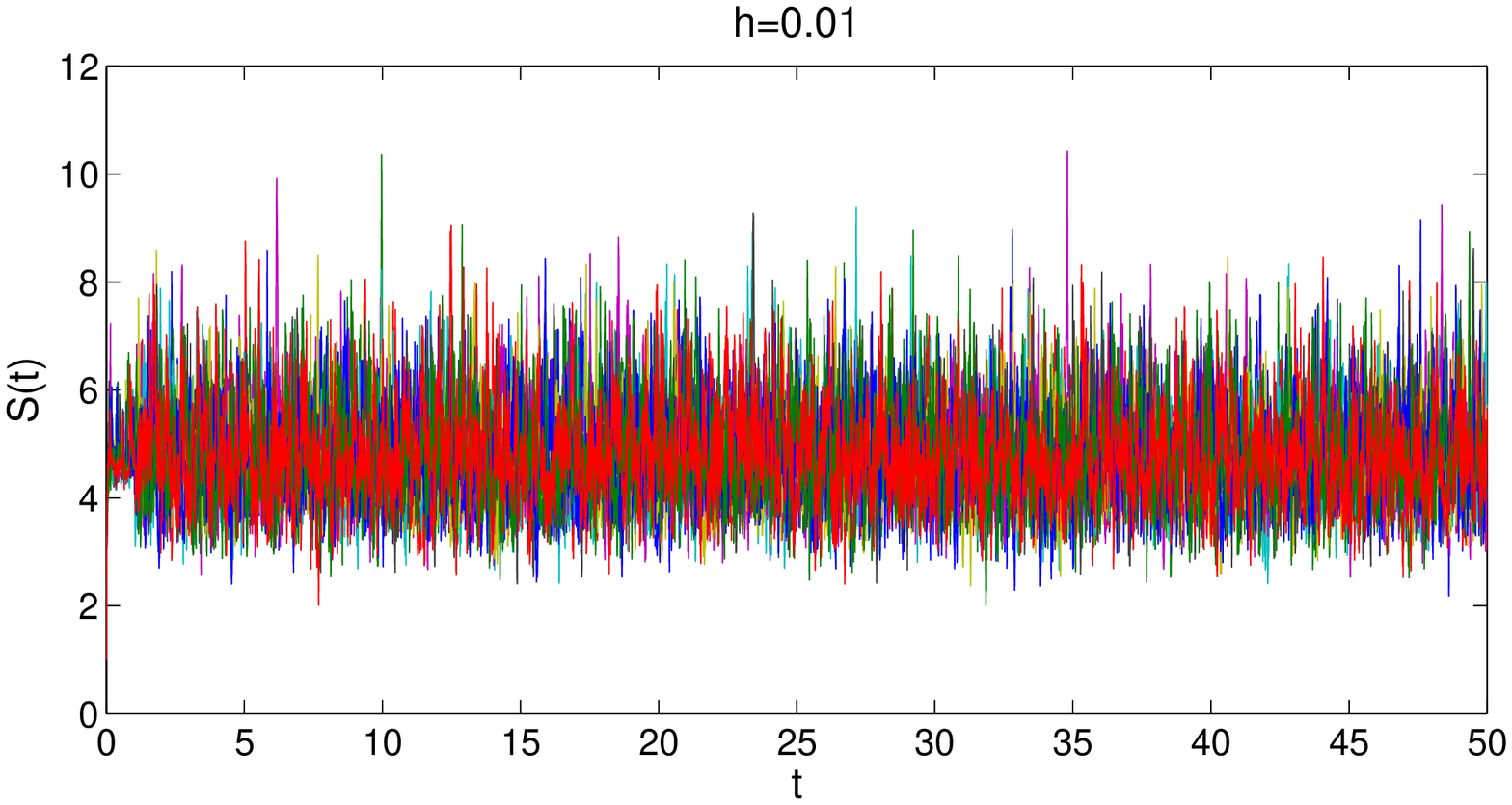}
\caption{Ten solution paths of
 Example 2, approximated by BIM  with $C_0=200,\,C_2=5$ and $C_1=\sigma \frac{s_{n-m}^{\gamma}}{\sqrt{s_n}}.$ }
\label{Fj3}
\includegraphics[width=12cm]{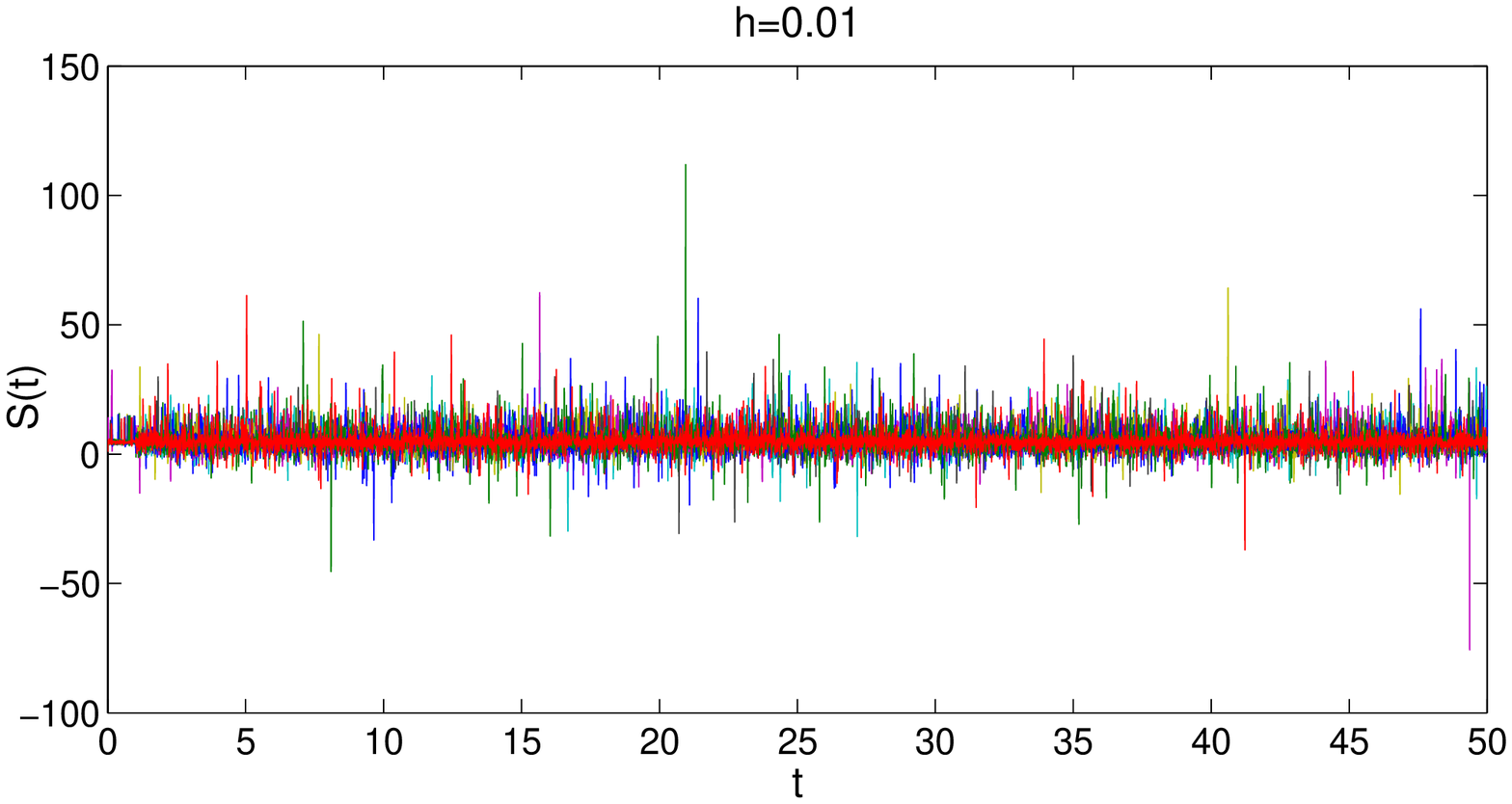}
\caption{Ten solution paths of
 Example 2, approximated by Euler method. }
\label{Fj4}
\end{center}
\end{figure}
\begin{figure}[h!]
\begin{center}
\includegraphics[width=12cm]{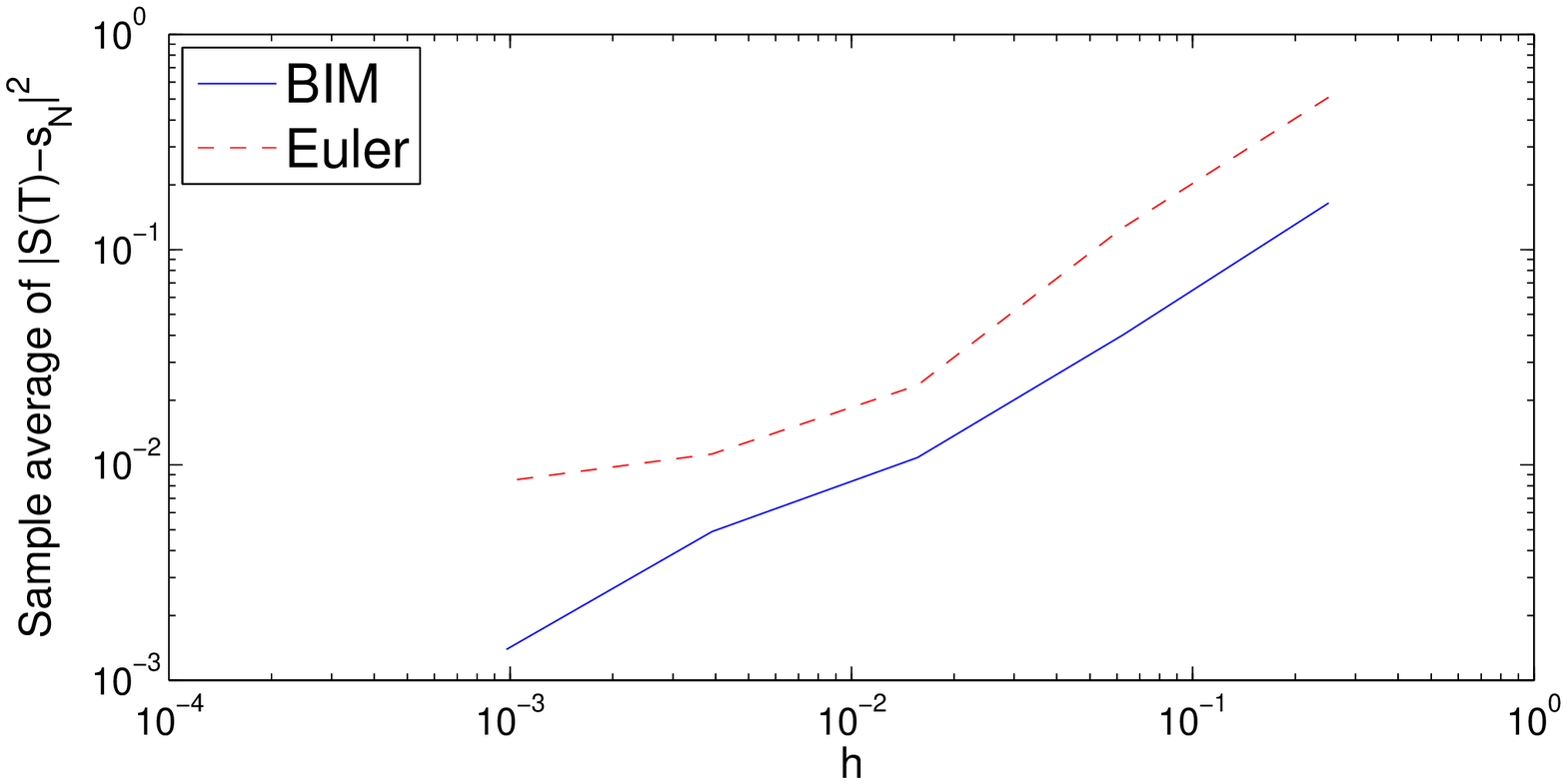}
\caption{Strong error of the  BIM   with   $C_0=10,\,C_2=1,$ $C_1=\sigma \frac{s_{n-m}^{\gamma}}{\sqrt{s_n}},$
  applied to Example 1.  }\label{Fj5}
\end{center}
\end{figure}
\begin{figure}[h!]
\begin{center}
\includegraphics[width=12cm]{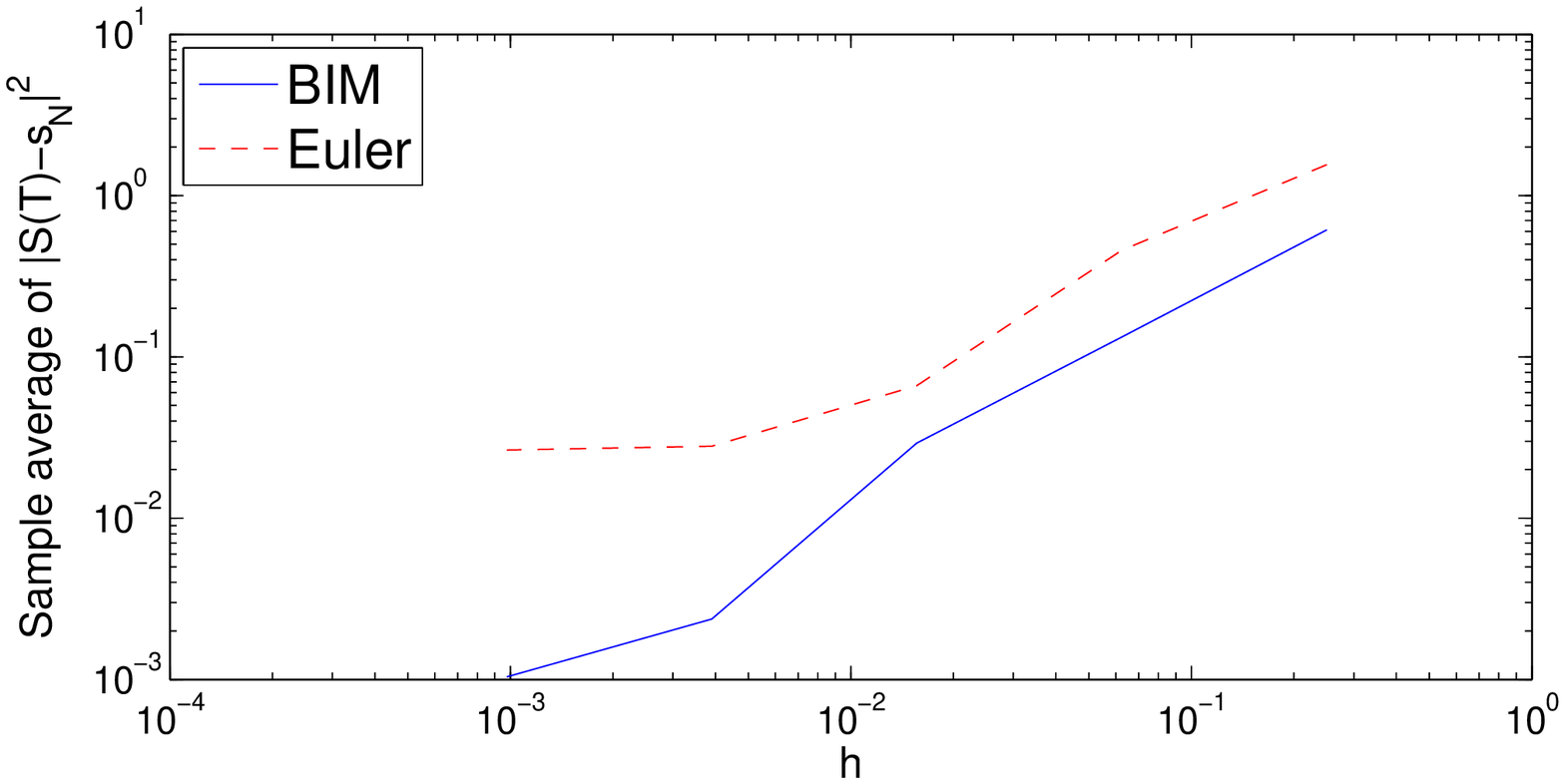}
\caption{Strong error of the  BIM   with   $C_0=200,\,C_2=5,$ $C_1=\sigma \frac{s_{n-m}^{\gamma}}{\sqrt{s_n}},$
  applied to Example 2.  }\label{Fj6}
\end{center}
\end{figure}
\begin{figure}[h!]
\begin{center}
\includegraphics[width=12cm]{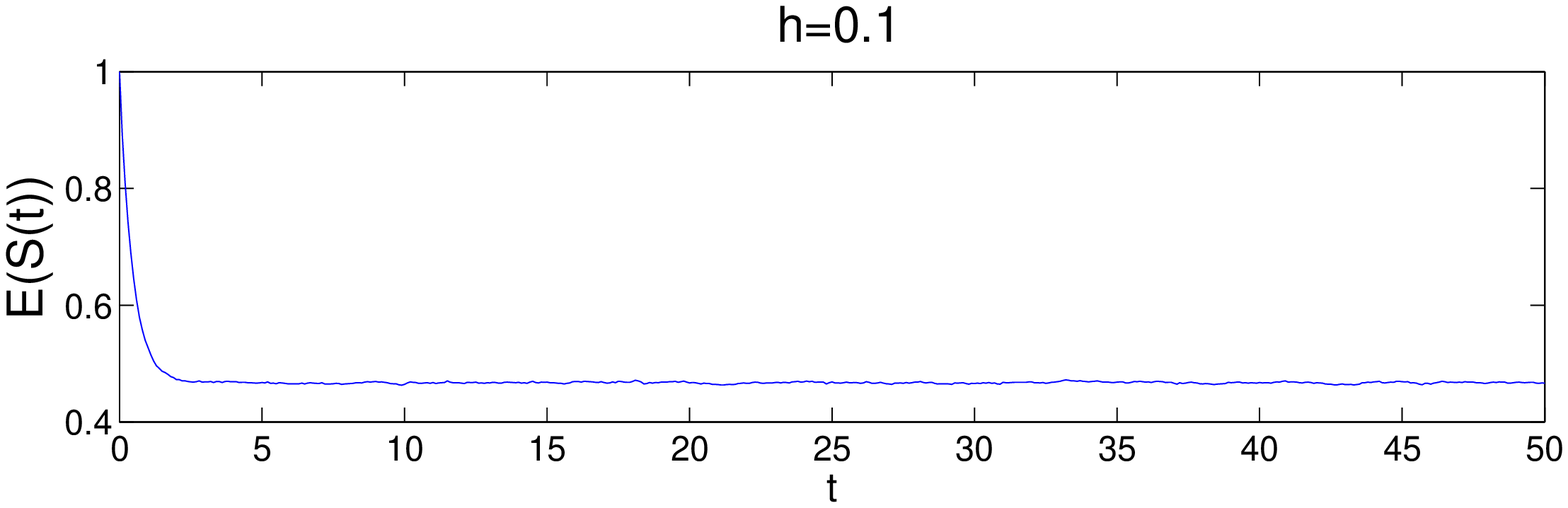}
\caption{$E(S(t))$ vs. $t$ of  the BIM  with $C_0=10,\,C_2=1$ and $C_1=\sigma
\frac{s_{n-m}^{\gamma}}{\sqrt{s_n}}$
and with step size $h=0.1,$ for
the Example 1.} \label{Fsj1}
\includegraphics[width=12cm]{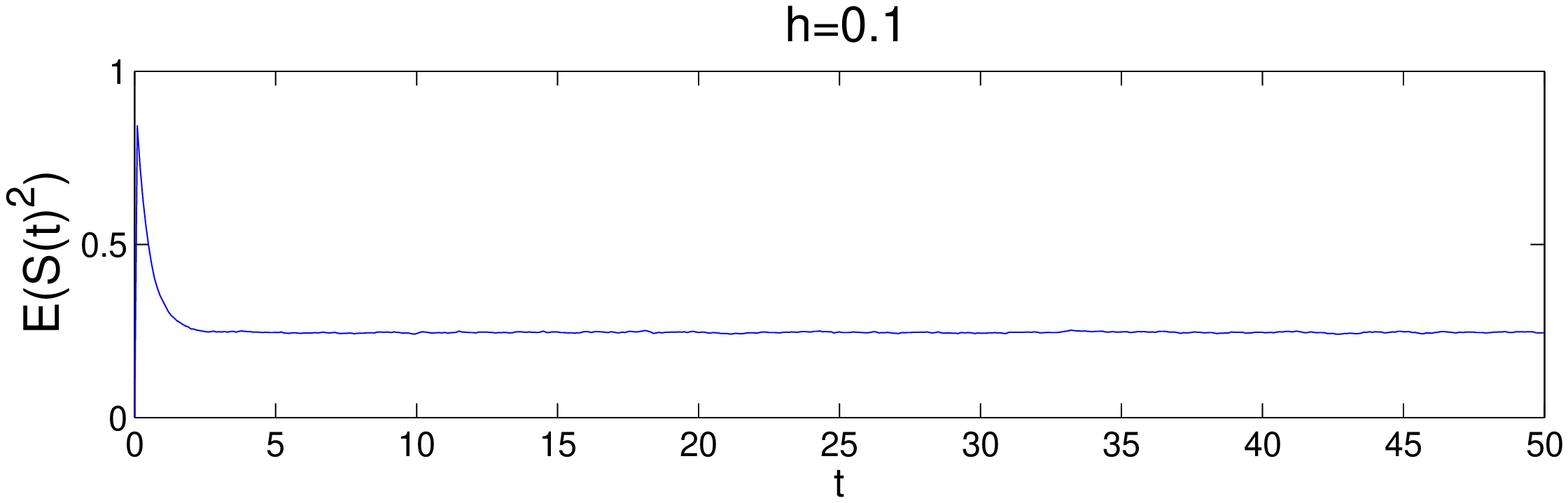}
\caption{$E(S(t)^2)$ vs. $t$ of  the BIM  with $C_0=10,\,C_2=1$ and $C_1=\sigma
\frac{s_{n-m}^{\gamma}}{\sqrt{s_n}}$
and with step size $h=0.1,$ for
the Example 1.} \label{Fsj3}
\includegraphics[width=12cm]{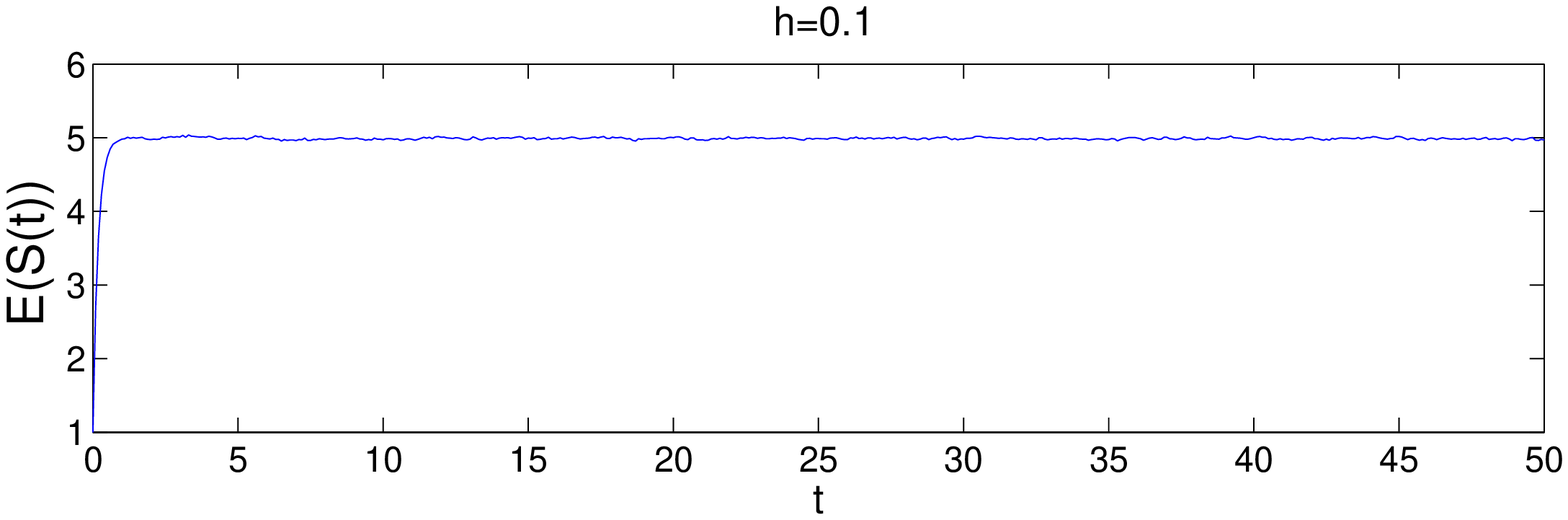}
\caption{$E(S(t))$ vs. $t$ of  the BIM  with $C_0=200,\,C_2=5$ and $C_1=\sigma
\frac{s_{n-m}^{\gamma}}{\sqrt{s_n}}$
and with step size $h=0.1,$ for
the Example 2.} \label{Fsj2}
\includegraphics[width=12cm]{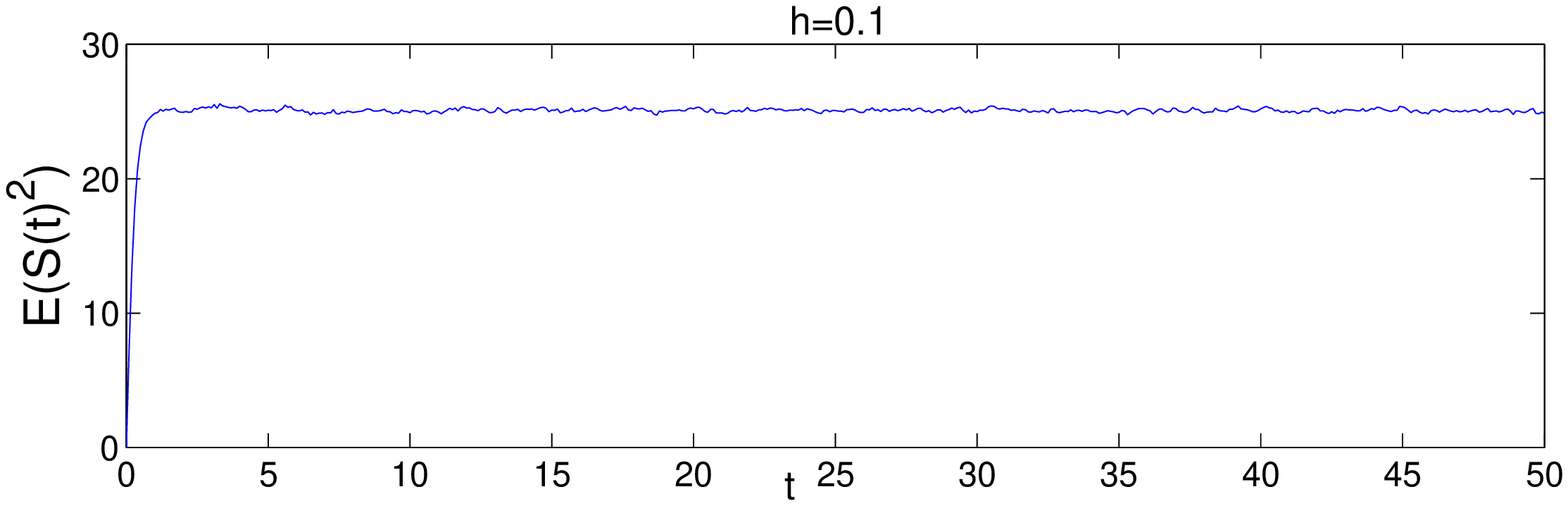}
\caption{$E(S(t)^2)$ vs. $t$ of  the BIM  with $C_0=200,\,C_2=5$ and $C_1=\sigma
\frac{s_{n-m}^{\gamma}}{\sqrt{s_n}}$
and with step size $h=0.1,$ for
the Example 2.} \label{Fsj4}
\end{center}
\end{figure}
\section{Conclusions}
In this work, we have demonstrated convergence and non-negativity properties of the
numerical solution obtained by the BIM for delay
 CIR model with jump.  First, we have chosen control functions of the BIM  such that this method can preserve
  non-negativity of  solution of the model. Then, we have studied the    moments boundedness and convergence
  of the solution of BIM by the determined control functions. Some numerical experiments have been included which illustrate the theoretical results
  obtained in this paper.

\end{document}